\documentclass[11pt]{amsart}

\usepackage{graphicx}
\usepackage{framed}

\usepackage{dsfont}
\usepackage[utf8]{inputenc}	

\usepackage[all]{xy}
\usepackage{pb-diagram, pb-xy}

\usepackage{url}

\usepackage{tikz}
\usepackage{pgffor}
\usetikzlibrary{arrows,backgrounds}
\usetikzlibrary{calc}

\usepackage{amssymb}
\usepackage{amsthm}

\usepackage{hyperref}

\usepackage{calc}

\newtheoremstyle{plain}         % name
  {1\baselineskip}              % Space above, empty = `usual value'
  {1\baselineskip}              % Space below
  {\slshape}                    % Body font
  {}                            % Indent amount (empty = no indent, \parindent = para indent)
  {\bfseries}                   % Thm head font
  {.}                           % Punctuation after thm head
  {.5em}                        % Space after thm head: " " = normal interword space;
  {\thmnumber{#2 }\thmname{#1}\thmnote{ \normalfont\textsc{(#3)}}} % Thm head spec

\theoremstyle{plain}
\newtheorem{thm}{Theorem}[section]

\newtheorem{prop}[thm]{Proposition}
\newtheorem{cor}[thm]{Corollary}
\newtheorem{lem}[thm]{Lemma}
\newtheorem{dfn}[thm]{Definition}

\newtheoremstyle{special}       % name
  {1\baselineskip}              % Space above, empty = `usual value'
  {1\baselineskip}              % Space below
  {\slshape}                    % Body font
  {}                            % Indent amount (empty = no indent, \parindent = para indent)
  {\bfseries}                   % Thm head font
  {.}                           % Punctuation after thm head
  {.5em}                        % Space after thm head: " " = normal interword space;
  {\thmnumber{#2 }\thmnote{#3}} % Thm head spec

\theoremstyle{special}

\newtheoremstyle{note}          % name
  {1\baselineskip}              % Space above, empty = `usual value'
  {1\baselineskip}              % Space below
  {\normalfont}                 % Body font
  {}                            % Indent amount (empty = no indent, \parindent = para indent)
  {\bfseries}                   % Thm head font
  {:}                           % Punctuation after thm head
  {.5em}                        % Space after thm head: " " = normal interword space;
  {\thmnumber{#2 }\thmname{#1}\thmnote{#3}}    % Thm head spec

\theoremstyle{note}

\newtheorem*{bem}{Remark}
\newtheorem{bem*}[thm]{Bemerkung}
\newtheorem*{bsp}{Example}

\def\Z{\mathbb{Z}}  %ersetzt durch \Z

\def\1{\mathds{1}}

\newcommand{\A}{\mathcal{A}}
\newcommand{\NN}{\mathcal{N}}

\def\Eins{\mathbb{1}}

\def\Eins{\mathds{1}}
\def\A{\mathcal{A}}

\newcommand{\bbG}{{\mathbb G}}

\newcommand{\bbZ}{{\mathbb Z}}

\newcommand{\calC}{{\mathcal C}}

\newcommand{\calN}{{\mathcal N}}

\newcommand{\calR}{{\mathcal R}}

\newcommand{\g}{\mathfrak{g}}

\usepackage{xypic,dsfont}

\hypersetup{
    colorlinks,
    citecolor=red,
    filecolor=black,
    linkcolor=blue,
    urlcolor=black
}

\begin{document}

\sloppy

\thispagestyle{empty}
\title[Semisimplified representation categories]{On supergroups and their semisimplified representation categories}
\author{Thorsten Heidersdorf}
\address{T.H.: Max-Planck Institut f\"ur Mathematik, Bonn}
\email{heidersdorf.thorsten@gmail.com}

%\begin{center}
%\textbf{PRELIMINARY VERSION}
%\end{center}

%\vspace{-1.1cm}

\begin{abstract} The representation category $\mathcal{A} = Rep(G,\epsilon)$ of a supergroup scheme $G$ has a largest proper tensor ideal, the ideal $\calN$ of negligible morphisms. If we divide $\mathcal{A}$ by $\calN$ we get the semisimple representation category of a pro-reductive supergroup scheme $G^{red}$. We list some of its properties and determine $G^{red}$ in the case $GL(m|1)$.
\end{abstract}


\maketitle


\thispagestyle{empty}
%\setcounter{tocdepth}{1}
%\tableofcontents \thispagestyle{empty}
%\addcontentsline{toc}{chapter}{Table of contents}

\setcounter{secnumdepth}{3} %Danach normale Nummerierungen (mehr als drei wird hier nicht verwendet)

\section{Introduction} 

A fundamental fact about finite-dimensional algebraic representations of a reductive group over an algebraically closed field $k$ of characteristic 0 is complete reducibility: Every representation decomposes into a direct sum of irreducible representations. This is no longer true if we consider representations of supergroups on super vector spaces. Indeed by a classical result of Djokovic-Hochschild \cite{Djokovic-Hochschild} the representation category of a Lie superalgebra $\g$ is semisimple if and only if $\g$ is a semisimple Lie algebra or of the form $\mathfrak{osp}(1|2n)$ for $n \geq 1$. Correspondingly many standard techniques from Lie theory do not work for representations of supergroups. Although a lot of progress has been made on representations of special supergroups such as $GL(m|n)$ and $OSp(m|2n)$, many classical questions are still open, most notably the tensor product decompositon of two irreducible representations. The category of finite-dimensional super representations $Rep(G)$ (or its full subcategory $Rep(G,\epsilon)$), $G$ a supergroup, is a tensor category. Every $k$-linear tensor category has a largest proper tensor ideal $\NN$, the tensor ideal of negligible morphisms. By \cite{Andre-Kahn} the quotient category $\omega: Rep(G,\epsilon) \to Rep(G,\epsilon)/\NN$ is an abelian semisimple $k$-linear tensor category. This quotient is called the semisimplification of $Rep(G,\epsilon)$. 

\medskip
Semisimplifications of other tensor categories have been studied in a variety of cases: A well-known example is the quotient of the category of tilting modules by the negligible modules (of quantum dimension 0) in the representation category of the Lusztig quantum group $U_q(\mathfrak{g})$ where $\mathfrak{g}$ is a semisimple Lie algebra over $k$ \cite{Andersen-Paradowski} \cite{Bakalov-Kirillov}. In \cite{Jannsen} Jannsen proved that the category of numerical motives as defined via algebraic correspondences modulo numerical equivalence is an abelian semisimple category. It was noted by Andr\'e and Kahn \cite{Andre-Kahn} that taking numerical equivalence amounts to taking the quotient by the negligible morphisms. A generalization of these results was obtained in \cite{Marcolli-Tabuada}. Recently Etingof and Ostrik \cite{Etingof-Ostrik} studied semisimplifactions with an emphasis on finite tensor categories. 

\begin{thm} (Theorem \ref{prop:fundamental}) The quotient $Rep(G,\epsilon)/\NN$ is a super-tannakian category, i.e. it is of the form $Rep(G^{red},\epsilon')$ where $G^{red}$ is a supergroup scheme with semisimple representation category and $\epsilon': \mu_2 \to G$ such that the operation of $\mu_2$ gives the $\Z_2$-graduation of the representations.
\end{thm}

This result follows immediately from a characterization of representation categories due to Deligne \cite{Deligne-tensorielles}. A natural question is to understand and possibly determine $G^{red}$ for given $G$. This is very difficult and not even possible in the general case. We assemble a few general results about these quotients and then focus on the $GL(m|n)$-case ($m \geq n$).

\medskip

We show that the classification of the irreducible representations of $G^{red}$ is a wild problem for $n \geq 3$ in theorem \ref{thm:wildness}. Hence the question should be modified as follows: We should study the subcategory in $Rep(G^{red},\epsilon')$ generated by the images $\omega(L(\lambda))$ of the irreducible representations of $G$. To determine this subcategory would amount to determine the tensor product decomposition of irreducible representations up to superdimension 0 and would give a parametrization of the indecomposable summands of non-vanishing superdimension. We study this problem in \cite{Heidersdorf-Weissauer-tannaka} in the case of $GL(n|n)$. The cases $GL(m|1)$ and $SL(m|1)$ are rather special since the blocks are of tame representation type and the indecomposable representations have been classified \cite{Germoni-sl} and we can hope to determine the entire quotient category. From the classification it is easy to determine the irreducible objects of  $Rep(G^{red},\epsilon')$ in lemma \ref{thm:irreducible}. We then compute their tensor product decomposition in theorem \ref{thm:tensor-product}. 

\begin{thm} (Theorem \ref{thm:main}) The quotient $Rep(GL(m|1))/\mathcal{N}$ is equivalent to the super representations of $GL(m-1) \times GL(1) \times GL(1)$ and therefore \[ GL(m|1)^{red} \cong \Z/2\Z  \ltimes ( \ GL(m-1) \times GL(1) \times GL(1) \ )  \] where $\Z/2\Z$ acts by the parity automorphism on $GL(m-1) \times GL(1) \times GL(1)$. 
\end{thm} 

In this statement we adopt the convention - as in the rest of the article - that any group scheme $G$ is seen as a supergroup scheme with trivial odd part and $Rep(G)$ is the category of super representations. In order to determine the tensor product decomposition we use two tools: The theory of mixed tensors \cite{Heidersdorf-mixed-tensors} gives us the tensor product decomposition between the irreducible $GL(m|1)$-representations. We then use cohomological tensor functors $DS: Rep(GL(m|1)) \to Rep(GL(m-1))$ akin to those of \cite{Duflo-Serganova} \cite{Heidersdorf-Weissauer-tensor} to reduce the tensor product decomposition between indecomposables to the irreducible case. The main point here is that $DS(V)$ is a $\Z$-graded object for any $V$, hence $DS$ could be interpreted as a functor to $\Z \times Rep(GL(m-1))$.

\medskip
To determine $G^{red}$ is probably in reach for the simple supergroups of maximal atypicality 1. To determine the subgroup of $G^{red}$ corresponding to the irreducible representations is already very difficult for $GL(m|n)$ and even more so for the supergroups $OSp(m|2n)$ and $P(n)$.

\subsection*{Acknowledgements} I would like to thank the referee for providing many useful comments on an earlier version of this paper.

\medskip

\section{Preliminaries}\label{sec:preliminaries}

\textit{Super linear algebra}. Throughout the article $k$ is an algebraically closed field of characteristic 0. A super vector space is a finite-dimensional $\Z_2$-graded vector space $V = V_0 \oplus V_1$ over $k$. Elements in $V_0$ respectively $V_1$ are called even respectively odd.  An element is homogenuous if it is either even or odd. For a homogeneous element $v$ write $p(v)$ for the parity defined by \begin{align*} p(v) = \begin{cases} 0 & v \in V_0 \\ 1 &  v \in V_1. \end{cases}\end{align*}  We denote by $Hom(V,W)$ the set of $k$-linear parity-preserving morphism between two super vector spaces $V$ and $W$. The parity shift functor $\Pi:svec \to svec$ is defined by $(\Pi V)_0 = V_1$, $(\Pi V)_1 = V_0$ and on morphisms $f:V \to W$ via $\Pi f: v \mapsto f(v)$ where $v$ is viewed as an element of $\Pi V$ and $f(v)$ as an element of $\Pi W$. 

\medskip
We refer the reader to \cite{Serganova-quasireductive} \cite{Westra} for some background on affine supergroup schemes. Recall that a functor from the category of commutative superalgebras $G:salg \to sets$ is called a supergroup scheme $G$ if $G$ is a representable group functor, i.e. there is some commutative Hopf superalgebra $R$ such that $G \cong Spec(R)$. It is called a supergroup if the representing superalgabra is finitely generated. A particular important example is the functor $G = GL(m|n)$ from the category of commutative superalgebras to the category of groups which sends $A = A_0 \oplus A_1$ to the invertible $(m+n) \times (m+n)$ matrices of the form \[ \begin{pmatrix} a & b \\ c & d \end{pmatrix} \] where $a$ is an $(m \times m)$-matrix with entries in $A_0$,  $b$ is an $(m \times n)$-matrix with entries in $A_1$, $c$ is an $(n \times m)$-matrix with entries in $A_1$ and $d$ is an $(n \times n)$-matrix with entries in $A_0$.

\medskip
{\it Representations}. We denote the category of finite-dimensional representations of $G$ on super vector spaces with parity preserving morphisms by $Rep(G)$. Let $G$ be a supergroup scheme and let $\epsilon$ be an element of $G(k)$ of order dividing 2 such that the automorphism $int(\epsilon)$ of $G$ is the parity automorphism defined by $x \mapsto (-1)^{p(x)}x$ for homogeneous $x$. Then let $Rep(G,\epsilon)$ be the category of (finite-dimensional) representations $V = (V,\rho)$ such that $\rho(\epsilon)$ is the parity automorphism of $V$. 

\begin{bsp} \cite[Example 0.4 (i)]{Deligne-tensorielles} If $G$ is an affine group scheme, then $\epsilon$ is central. If $\epsilon$ is trivial, one recovers the ordinary representation category of $G$ (non-super). If $\epsilon$ is non-trivial, for every representation $(V,\rho)$ of $G$ the involution $\rho(\epsilon)$ defines a $\Z_2$-graduation on $V$ and the commutativity isomorphism of the tensor product is given by the Koszul rule.  The category $Rep(G,\epsilon)$ identifies itself with $Rep(G,1)$  as a $k$-linear monoidal category, but they are not equivalent as symmetric monoidal categories.
\end{bsp}

\begin{bsp} \cite[Example 0.4 (ii)]{Deligne-tensorielles} If $G$ is an affine supergroup scheme, let $\mu_2$ act on $G$ by the parity automorphism. Then $Rep( \mu_2 \ltimes G, \epsilon = (-1,e))$ is the category of super representations of $G$. 
\end{bsp}

For the supergroup $GL(m|n)$ and $\epsilon = diag(E_m,-E_n)$ we put $Rep(GL(m|n),\epsilon) = \calR_{m|n}$. For the whole category $Rep(GL(m|n))$ we also write $\mathcal{T}_{m|n}$. Then $\mathcal{T}_{m|n} = \calR_{m|n} \oplus \Pi \calR_{m|n}$ \cite[Corollary 4.44]{Brundan-Kazhdan}.

\medskip
The categories $\mathcal{T}_{m|n}$ and $\calR_{m|n}$ are examples of super tannakian categories. For background on tensor categories we refer to \cite{Deligne-Milne}. We denote the unit object of a tensor category by $\Eins$.

\begin{dfn} A  $k$-linear, abelian, rigid tensor category $\A$ with $k \simeq End(\Eins)$ and with a $k$-linear exact faithful tensor functor $\rho: \A \to svec$ (a super fibre functor) is called a super tannakian category.
\end{dfn}

\begin{thm}\cite[8.19]{Deligne-Festschrift} Every super tannakian category $\A$ is tensor equivalent to the category  $\A \simeq Rep(G,\epsilon)$ of representations of a supergroup scheme $G$. 
\end{thm}

For a partition $\lambda$ of $n$ let $S_{\lambda}(-)$ be the associated Schur functor \cite{Deligne-tensorielles}. We put $Sym^n(X) = S_{(n)}(X)$ (the $n$-th symmetric power) and $\Lambda^n(X) = S_{(1,\ldots,1)}(X)$ (the $n$-th alternating power). An object $X$ of $\A$ is called Schur-finite if there exists an integer $n$ and a partition $\lambda$ of $n$ such that $S_{\lambda}(X) = 0$. 

\begin{thm}\cite[Th\'eor\`eme 0.6]{Deligne-tensorielles} \label{deligne-super} If $\A$ is an abelian $k$-linear rigid tensor category with $End(\Eins) \simeq k$ such that every object is Schur finite, then $\A$ is a super tannakian category, i.e. $\A \simeq Rep(G,\epsilon)$ for some supergroup scheme  $G$.
\end{thm}

\begin{lem} \label{germoni-nice} \cite[Lemma 1.1.1]{Germoni-sl} Let $\A$ be a small abelian k-linear category such that morphism spaces are finite-dimensional, every object has a  finite composition series and the category has enough projectives. Then
\begin{enumerate}
	\item The endomorphism ring of any indecomposable object is a local ring.
	\item Every object can be written as a direct sum of indecomposable objects.

	\item Every module has a unique projective cover

\end{enumerate}
\end{lem}

An example is given by the category $\calR_{m|n}$ \cite{Germoni-sl}\cite{Serganova-quasireductive}.

%---------------------------------

%--------------------------------

\section{The universal semisimple quotient}\label{sec:semisimple}

\subsection{Negligible morphisms}

An additive category $\A$ is a Krull-Schmidt category if every object has a decomposition in a finite direct sum of elements with local endomorphism rings. An ideal in a $k$-linear category is for any two objects $X,Y$ the specification of a $k$-submodule $\mathcal{T}(X,Y)$ of $Hom_{\A}(X,Y)$, such that for all pairs of morphisms $f \in Hom_{\A}(X,X'), \ g \in Hom_{\A}(Y,Y')$ the inclusion $g \mathcal{T}(X',Y)f \ \subseteq \mathcal{T}(X,Y')$ holds. Let $\mathcal{T}$ be an ideal in $\A$. Then $\A/{\mathcal{T}}$ is the category with the same objects as $\A$ and with $Hom_{\A/\mathcal{T}} (X,Y ) = Hom_{\A}(X,Y)/\mathcal{T}(X,Y)$. It is again a Krull-Schmidt category \cite{Liu}, \cite[Lemma 2.1]{Koenig-Zhu}. Suppose that $\A$ is abelian and that every object has finite length and let $X$ be an indecomposable element and $\phi$ an endomorphism of $X$. By Fitting's lemma $\phi$ is either invertible or nilpotent. An element $X$ is indecomposable if and only if its endomorphism ring is a local ring. We assume in the following that $\A$ is a super tannakian category or a pseudoabelian full tensor subcategory. Then all the above conditions hold. 

\medskip
An ideal in a tensor category is a tensor ideal if it is stable under $id_C \otimes -$ and $- \otimes id_C$ for all $C \in \A$. The ideal is then stable under tensor products from left or right with arbitrary morphisms. Let $Tr$ be the trace. For any two objects $A, B$ we define $\mathcal{N}(A,B) \subset Hom(A,B)$ by \[ \mathcal{N}(A,B) = \{ f \in Hom(A,B) \ | \ \text{ for all } g \in Hom(B,A), \ Tr(g \circ f ) = 0 \}. \] The collection of all $\mathcal{N}(A,B)$ defines a tensor ideal $\mathcal{N}$ of $\mathcal{A}$ \cite{Andre-Kahn}, the tensor ideal of negligible morphisms. By \cite[Th\'eor\`eme 8.2.2a]{Andre-Kahn} we have the following theorem.

\begin{thm} (i) $\mathcal{N}$ is the largest proper tensor ideal of $\A$.\\
(ii) The only proper tensor ideal $\mathcal{I}$ of $\A$ such that the quotient  $\A/\mathcal{I}$ is semisimple, is $\mathcal{I} = \mathcal{N}$. 
\end{thm}

The quotient $\A/\NN$ will be called the universal semisimple quotient of $\A$. 

\begin{thm} \label{prop:fundamental} The quotient $\A/\mathcal{N}$ is again a super tannakian category. If $\A' \subset \A$ is a pseudoabelian full tensor subcategory, the quotient $\A'/(\NN \cap \A')$ is a super tannakian category.
\end{thm}

\begin{proof} The quotient of a $k$-linear rigid tensor category by a tensor ideal is again a $k$-linear rigid tensor category. Since $\NN$ is a tensor ideal the quotient functor $\omega: \A \to \A/\NN$ is a tensor functor. The quotient category is semisimple by construction. Since $Hom$-spaces are finite-dimensional one has idempotent lifting, hence $\A/\NN$ is pseudoabelian. A $k$-linear semisimple pseudoabelian category is abelian by \cite[Proposition 2.1.2]{Andre-Kahn}. By \cite[Th\'eor\`eme 0.6]{Deligne-tensorielles} an abelian tensor category is super tannakian if and only if for every object $A$ there exists a Schur functor $S_{\mu}$ with $S_{\mu}(A) = 0$. Since $\omega(S_{\mu}(A)) = S_{\mu}(\omega(A))$ any object in $\A/\NN$ is also annulated by a Schur functor.  
\end{proof}

\begin{bem} Andr\'e and Kahn prove in the tannakian case a stronger statement which does not use that $k$ is algebraically closed. They show \cite[Th\'eor\`eme 13.2.1]{Andre-Kahn} that there exists a tensorial splitting \[ \xymatrix{ \A \ar[r]^{\omega} & \A/\NN \ar@/^1pc/[l]^{s}  }   \] If $\rho: \A \to vec$ is a fiber functor, the composition $\rho \circ s$ is a fiber functor for $\A/\NN$. In particular $\A/\NN$ is a neutral tannakian category if $\A$ is one, and $\A/\NN$ is the representation category of a group scheme. There cannot be such a splitting in the super tannakian case since $\omega(X)$ is zero for any indecomposable $X$ of superdimension 0 by lemma \ref{object-zero}. However the question remains if $\A/\NN$ is neutral for a neutral super tannakian category $\A$.
\end{bem}

The category $\A/\NN$ has the following universal property.

\begin{prop}\label{prop:factorization} Let $\omega: \A \to \mathcal{C}$ be a full tensor functor into a semisimple tensor category $\mathcal{C}$. Then $\omega$ factorises over the quotient $\A/\NN$.
\end{prop}

\begin{proof} Since $\mathcal{C}$ is semisimple there are no negligible morphisms. However the image of a negligible morphism is negligible, since a tensor functor commutes with traces. Hence the image of a negligible morphism under $\rho$ is zero, hence the functor factorizes. 
\end{proof}

For completeness sake we assemble a few elementary lemmas about this quotient.

\begin{lem} An indecomposable object $X$ of $\A$ maps to zero in $\A/\NN$ if and only if $id_X$ belongs to $\mathcal{N}(X,X)$.
\end{lem}

The collection of these objects - called negligible objects - is denoted by $N$. 

\medskip
The dimension of an object $X$ in a tensor category is defined $Tr(id_X) \in End(\Eins)$. If $\A \subseteq Rep(G,\epsilon)$, then $dim_{\A}(X) = sdim(X) = dim(X_0) - dim(X_1)$.

\begin{lem} \label{object-zero} An indecomposable object is in $N$ if and only if $sdim(X) = 0$.
\end{lem}

\begin{proof} If $X \in N$ we have $Tr(g) = 0$ for all $g \in End(X)$, in particular for $g= id_X$. Let $sdim(X) = 0$. We have to show: $id_X \in N(X,X)$, ie. $Tr(g) = 0$ for all $g \in End(X)$. Since $X$ is indecomposable $g$ is either nilpotent or an isomorphism. If $g$ is nilpotent $Tr(g) = 0$ \cite[Lemme 1.4.3]{Bruguieres}. Let $g$ be an isomorphism. Since $X$ is indecomposable $g$ has a unique eigenvalue $\lambda$ and $Tr(g) = \lambda sdim(X)$, hence  $Tr(g) =0$.   
\end{proof}

\begin{lem} (\cite[Section 1.4]{Bruguieres}) The functor $\A \to \A/\mathcal{N}$ induces a bijection between the isomorphism classes of indecomposable elements not in $N$ and the isomorphism classes of irreducible elements in $\A/\NN$. 
\end{lem}

\begin{proof} Let $X$ be indecomposable, $X \notin N$. Since $\A$ and $\A/\NN$ are abelian and every object has finite lenght, an object $X$ is indecomposable if and only if $End(X)$ is a local ring. We have $End_{\A/\NN}(X) = End_{\A}(X) / \NN(X)$. Since the quotient of a local ring by a (two-sided) ideal is again local, the image of $X$ in $\A/\NN$ is indecomposable, hence irreducible.  We show: If $M \ncong N$ in $\A$ ($M,N$ indecomposable) we have $Hom_{\A/\NN}(M,N) = 0$. Let $f \in Hom_{\A}(M,N)$. Its image is zero in $Hom_{\A/\NN}(M,N) = Hom_{\A}(M,N) / \NN(M,N)$ if and only if $Tr(fg) = 0$ for all $g \in Hom_{\A}(N,M)$. Since $M$ is indecomposable any endomorphism is invertible or nilpotent. The endomorphism $fg$ is not bijective, hence nilpotent, hence $Tr(fg) = 0$ for all $g \in Hom(N,M)$, hence $Hom_{\A/\NN}(M,N) = 0$.
\end{proof}

Let $I$ be an ideal in $\A$. For $X = \bigoplus X_i$ and $Y= \bigoplus Y_j$ we have canonically $I(X,Y) = \bigoplus_{i,j} I(X_i,Y_j)$ by \cite{Andre-Kahn}. Let $X = \bigoplus X_i$ with $X_i \in N$ for all $i$, ie. $\NN(X_i,Y) = Hom(X_i,Y)$ and $\NN(Y,X_i) = Hom(Y,X_i)$ for all $Y \in \A$. It follows $\NN(X,X) = Hom(X,X)$, hence $X \in N$. If reciprocally $X \in N$ and $X = \bigoplus X_i$, we have $X_i \in N$.

\begin{cor} (i) $N$ is closed under direct sums and direct summands.
(ii) If $X \in N$ and $Y \in \A$, we have $X \otimes Y$ in $N$ and each indecomposable summand of $X \otimes Y$ has superdimension $0$.
(iii) Let $X \notin N$ and let $X = \bigoplus X_i$ be its decomposition into indecomposable elements. Then $Hom_{\A/\NN}(X,X) = \bigoplus_{i, \ sdim(X_i) \neq 0} k$.
%(iv) $N$ is neither closed unter submodules nor quotients.
\end{cor}

\subsection{The pro-reductive envelope}

Since the quotient $\A/\mathcal{N}$ is again a super-tannakian category, this defines a reductive supergroup scheme $G^{red}$ with $\A/\NN \simeq Rep(G^{red},\epsilon')$ with $\epsilon': \mu_2 \to G$ such that the operation of $\mu_2$ gives the $\Z_2$-graduation of the representations. We call $G^{red}$ the pro-reductive envelope of $G$ (following \cite{Andre-Kahn}). If $G$ is  an algebraic group, the pro-reductive envelope has been extensively studied by Andre and Kahn. Their proofs do not apply to the supergroup case. In the tannakian case $\NN = \mathcal{R}$ is equal to the radical ideal. In particular no indecomposable objects map to zero. Even in the tannakian case the pro-reductive cover will not be of finite type in general.

\begin{thm} \cite{Andre-Kahn}, theorem C.5.  The proreductive envelope of an affine $k$-group $G$ is of finite type over $k$ if and only if $G$ is of finite type over $k$ and the prounipotent radical of $G$ is of dimension $\leq 1$.  
\end{thm}

\begin{bsp} a) If $G = \bbG_a$, then $G^{red} = SL(2)$. 
b) If $G = \bbG_a \times \bbG_a$, then $G^{red}$ is no longer of finite type. In fact, the determination of $G \hookrightarrow G^{red}$ is unsolvable since it would include a classification of the indecomposable representations of $G$ which is a wild problem \cite[19.7]{Andre-Kahn}. 
\end{bsp}

More generally $Rep(G^{red},\epsilon')$ is of finite or tame type if $Rep(G)$ is of finite or tame type. The converse is not so obvious: If $Rep(G)$ is of wild type, the problem of classifying indecomposable modules of non-vanishing superdimension will often be wild as well. However for the supergroup $Q(n)$ corresponding to the simple Lie superalgebra $\mathfrak{q}(n)$ every nontrivial irreducible representation has superdimension $0$ \cite[Theorem 1.2]{Cheng-queer}. While the classification of the indecomposable representations is of $Q(n)$ is wild, this might not be true for the ones of non-vanishing superdimension. 

If the problem of classifying indecomposable modules of non-vanishing superdimension is wild, we should not try to determine $G^{red}$ in this case, but ask the following weaker questions. Given any object $V \in Rep(G)$ or $Rep(G,\epsilon)$, consider its image in $\A/\NN$. The tensor category generated by it is a semisimple algebraic tensor category (since $\A/\NN$ is semisimple). The semisimple algebraic tensor categories in characteristic zero were classified in \cite[Theorem 6]{Weissauer-semisimple}:

\begin{thm} \label{semisimple-classification}
Any supergroup $G$ over $k$ such that $Rep(G)$ is semisimple is isomorphic to a semidirect product $ G' \vartriangleleft H$
of a reductive algebraic $k$-group $H$ with a product $ G' =   \prod_{r\geq 1} Spo(1|2r)^{n_r}$ of simple supergroups of $BC$-type, where the semidirect product is defined by an abstract  group homomorphism $p: \pi_0(H) \to \ \prod_{r\geq 1}\ S_{n_r}$.
\end{thm}

Now consider an irreducible object $V \in Rep(G,\epsilon)$ and consider the tensor category generated by $\omega(V)$ in $Rep(G,\epsilon)/\mathcal{N}$. This tensor subcategory corresponds to an algebraic supergroup $G_V \hookrightarrow G^{red}$ as in \ref{semisimple-classification}. Then this group is of finite type since it has a tensor generator. If $G_V$ is reductive (i.e. not containing $OSp(1|2n))$ we have the following weak a priori estimate.

\begin{lem} Let $T$ be a maximal torus in $G_V$ and $X^*(T)$ its character group. Let $R$ be the subgroup generated by the roots of $G_V$. Then the center of $G_V$ has cyclic character group $X/R$ and $(G_V)^0_{der}$ has cyclic center.
\end{lem}

\begin{proof} $\omega(V)$ is a tensor generator of $Rep(G_V)$ for the reductive group $G_V$ and likewise $\omega(V)$ is a tensor generator of $Rep(G_V)_{der}^0$. Now use that a reductive group has a faithful irreducible representation if and only if $X/R$ is cyclic and a semisimple group has a faithful irreducible representation if and only if its center is cyclic \cite{mathoverflow}. 
\end{proof}

\subsection{The basic classical cases} Let $G$ be basic classical \cite{Serganova-quasireductive} with underlying basic classical Lie superalgebra $\g$ \cite{Kac-Rep}. Duflo and Serganova \cite{Duflo-Serganova} and \cite{Serganova-kw} constructed for certain elements $x \in \mathfrak{g}_1$ with $[x,x] = 0$, where $\mathfrak{g}_1$ denotes the odd part of $\g$, tensor functors $V \mapsto V_x: Rep(\g) \to Rep(\g_x)$ where $\g_x$ is a classical Lie algebra or $\mathfrak{osp}(1|2n)$. These functors are not full, hence need not factorize over the quotient $Rep(G)/\NN$. However it should be expected that $G^{red}$ contains groups $G_x$ with Lie superalgebra $\g_x$. For instance the superdimension of any irreducible representation in $Rep(G)$ equals the superdimension of some representation in $Rep(G^{red})$ and in $Rep(\g_x)$. Note that this representation in $Rep(\g_x)$ might not be irreducible.

\medskip

For $\mathfrak{gl}(m|n)$ we have $\g_x = \mathfrak{gl}(|m-n|)$ and for $\mathfrak{osp}(m|2n)$, $m = 2l$ or $2l+1$, we have $\g_x =  \mathfrak{osp}(m - 2min(l,n), 2n - 2 min(l,n))$.  For the exceptional Lie superalgebras the functor of Duflo-Serganov gives representations of the following Lie algebras:

\begin{enumerate}
\item If $\g = D(2,1,\alpha)$, then $\g_x = \mathfrak{gl}(1)$.
\item If $\g = G_3$, then $\g_x =  \mathfrak{sl}(2)$.
\item If $\g = F_4$, then $\g_x = \mathfrak{sl}(3)$. 
\end{enumerate}

Hence $G^{red}$ should contain $GL(1)$ or $SL(2)$ or $SL(3)$ as a subgroup respectively. We determine $Rep(GL(m|1))/\NN$ in this article. For the $GL(n|n)$-case see \cite{Heidersdorf-Weissauer-tannaka}. The $OSp(2|2n)$-case can be treated similar to the $GL(m|1)$-case. In this case we obtain $Rep(OSp(2|2n))/\NN \simeq Rep(Sp(2n-2) \times GL(1) \times GL(1))$. 

\begin{bem} Analogs of $DS$ can also be defined in the $P(n)$ and $Q(n)$-case. For $P(n)$ it is currently not known which irreducible representations have non-vanishing superdimension. In the $Q(n)$-case every nontrivial irreducible representation has superdimension 0. This does not mean that the group $G^{red}$ is trivial since we can still have many indecomposable representations with non-vanishing superdimension.
\end{bem}

%%%%%%%%%%%%%%

\subsection{How to determine the semisimplification} In general it is quite hard to determine the semisimplification $\A/\NN$ or even the group $G_V$ corresponding to an indecomposable representation $V$. Here are some comments on possible approaches and their difficulties.

\begin{enumerate}
\item The functor $\omega$ is not exact, hence it does not induce a homomorphism between the Grothendieck rings $K_0(\A)$ and $K_0(\A/\NN)$. This would have allowed to compare the transcendence degrees of the Grothendieck rings and obtain estimates on the rank of the group $G^{red}$.

\item Taking the quotient modulo negligible objects is not compatible with restriction to subgroups. If $H \subset G$ is a subgroup, the corresponding restriction functor $Res: Rep(G) \to Rep(H)$ does in general not induce a functor $Rep(G)/\NN \to Rep(H)/\NN$. Although $Res$ preserves the superdimension of an indecomposable object $X$ of superdimension 0, $Res(X)$ will in general be a direct sum of indecomposable summands of positive and negative superdimensions whose superdimensions add up to zero.

\item \label{recognition} There is a vast literature on determining connected semisimple groups and a given representation from some discrete data. Examples are \cite{Larsen-Pink} (dimension data), Larsen's conjecture \cite{Guralnick-Tiep} and the classification of small representations \cite{Andreev-Vinberg-Elashvili}. All this requires the condition that the group is connected which we do not know for the pro-reductive envelope (see number \ref{torsion}). For the purpose of determining $G_V$ the most useful one might be \cite{Kraemer-Weissauer} which often determines $V$ and $G_V$ if $Sym^2(V)$ and $\Lambda^2(V)$ are either irreducible or irreducible plus trivial. However even this requires the computation of $\Lambda^2(V)$ and $Sym^2(V)$ which is in general difficult for an indecomposable representation of a supergroup. 

\item The functor $DS$, that is the crucial tool in the $GL(m|1)$-case, is not as powerful in more complicated cases. While $DS(L(\lambda))$ is known by \cite[Theorem 16.1]{Heidersdorf-Weissauer-tensor} in the $\mathcal{T}_{m|n}$-case, $DS$ is not exact and it is not known how $DS(X)$ behaves for indecomposable objects $X$ of lenght 2 in $\mathcal{T}_{m|n}$. Note also that $DS$ does not necessarily preserve negligible objects: the Kac module $K(\Eins) \in \mathcal{T}_{n|n}$ of superdimension zero decomposes under $DS$ into a direct sum of irreducible representations of non-vanishing superdimension \cite[Section 10]{Heidersdorf-Weissauer-tensor}.

\item \label{torsion} The most fundamental problem is that the groups $G_V$ might not be connected. This is not a problem in the tannakian case considered in \cite{Andre-Kahn}: If an algebraic group $G$ is connected, $G^{red}$ is connected as well \cite[Lemme 19.4.1]{Andre-Kahn}. No such statement is known in the super case. This means that there could be \textit{torsion up to negligible objects} in $Rep(G)$, e.g. there could be indecomposable objects $X$ of superdimension $1$ with $X^{\otimes r} \cong \Eins \oplus \text{negligible}$ (coming from $\pi_0(G^{red})$ or a subgroup). Since there is typically no condition or restriction on the negligible part, such objects are hard to exclude. If we do not know that the group is connected, this makes the criteria in number \ref{recognition} rather useless. 

\end{enumerate}

%%%%%%%%%%%%%%%%%%%%%%%%%%%%%%%%

%%%%%%%%%%%%%%%%%%%%%%%%%%%%%%%%%

\section{On the $GL(m|n)$-case}

\subsection{Preliminaries} Recall from section \ref{sec:preliminaries} that $\mathcal{T}_{m|n}$ denotes the category of finite dimensional algebraic representations of $GL(m|n)$ with parity preserving morphisms. We always assume $m \geq n$.
The subcategory ${\calR}_{m|n}$ is stable under the dualities ${}^\vee$ and $^*$ (the twisted dual \cite[Section 4]{Brundan-Kazhdan}). The irreducible representations in $\calR_{m|n}$ are parametrized by their highest weight with respect to the Borel subgroup of upper triangular matrices. A weight $\lambda=(\lambda_1,...,\lambda_m \ | \ \lambda_{m+1}, \cdots, \lambda_{m+n}) \in X^+$ of an irreducible representation in $\calR_{m|n}$ satisfies $\lambda_1 \geq \ldots \lambda_m$, $\lambda_{m+1} \geq \ldots \lambda_{m+n}$ with integer entries \cite[Section 4]{Brundan-Kazhdan}. The Berezin determinant defines a one dimensional representation $Ber$. Its weight is is given by $\lambda_i=1$ for $i=1,\ldots,m$ and $\lambda_{m+i}=-1$ for $i=1,..,n$.

\medskip
For each weight $\lambda \in X^+$ we have the irreducible representation $L(\lambda)$, its projective cover $P(\lambda)$ and the Kac module (or standard objects) $K(\lambda)$, the universal highest weight module in $\mathcal{R}_{m|n}$. The Kac module $K(\lambda)$ has irreducible top and socle, and the top is given by the irreducible representation $L(\lambda)$  \cite[Proposition 2.4]{Kac-Rep}. We say that a module is a Kac object if it has a filtration whose subquotients are Kac modules. The full subcategory of these modules is denoted $\mathcal{C}^+$. Similarly we have the category $\mathcal{C}^-$ of objects which have a filtration by the twisted duals $K(\lambda)^*$ (costandard objects). By \cite[Theorem 4.37]{Brundan-Kazhdan} $\mathcal{C}^+ \cap \mathcal{C}^- = Proj$ (every tilting module is projective).

\medskip
The irreducible representations in $\mathcal{T}_{m|n}$ are given by the $\{L(\lambda), \Pi L(\lambda) \ | \ \lambda \in X^+ \}$  where $\Pi$ denotes the parity shift.

\begin{bem} Since $\mathcal{T}_{m|n} = \calR_{m|n} \oplus \Pi \calR_{m|n}$ and $\mathcal{R}_{m|n}$ is closed under tensor products, we can often work in $\calR_{m|n}$. This is however not true in section \ref{determination-of-envelope} since the Duflo-Serganova functor defines a tensor functor $DS:\mathcal{R}_{m|n} \to \mathcal{T}_{m|n}$.
\end{bem}

\subsubsection{Weight diagrams.} To each highest weight $\lambda\in X^+$  we associate, following \cite[Section 1]{Brundan-Stroppel-4}, two subsets of cardinality $m$ respectively $n$ of the numberline $\Z$
\begin{align*} I_\times(\lambda)\ & =\ \{ \lambda_1  , \lambda_2 - 1, .... , \lambda_n - n +1 \} \\
 I_\circ(\lambda)\ & = \ \{ 1 - m - \lambda_{m+1}  , 2 - m - \lambda_{m+2} , .... ,  n-m- \lambda_{m+n}  \}. \end{align*}

The integers in $ I_\times(\lambda) \cap I_\circ(\lambda) $ are labeled by $\vee$, the remaining ones in $I_\times(\lambda)$ resp. $I_\circ(\lambda)$ are labeled by $\times$ resp. $\circ$. All other integers are labeled by a $\wedge$. This labeling of the numberline $\Z$ uniquely characterizes the weight $\lambda$. If the label $\vee$ occurs $r$ times in the labeling, then $r$ is called the degree of atypicality of $\lambda$. Notice that $0 \leq r \leq n$, and $\lambda$ is called maximal atypical if $r=n$. Examples are the trivial module $\Eins$ and the standard representation $V$ of highest weight $\lambda = (1,\ldots,0|0,\ldots,0)$ for $m\neq n$. Another example is the Berezin determinant $Ber = L(1,\ldots,1 \ | \ -1,\ldots,-1)$ of dimension $1$. For the notion of a cup or cap diagram attached to a weight diagram we refer to \cite[Section 2]{Brundan-Stroppel-1}\cite{Brundan-Stroppel-4}.

%%%%%%%%%%%%%%%%%%%%%%%%%%%%

\subsection{Wildness}

If $\A$ is a category as in lemma \ref{germoni-nice} we can associate to it its Ext-quiver. The vertex set $X^+$ is given by the set of isomorphism classes of simple modules and the number of arrows from $\lambda$ to $\mu$ is given by $Ext_{\A}^1(L(\lambda), L(\mu))$.

\begin{thm} \cite[Theorem 1.4.1]{Germoni-sl}. Let $\A$ be a nice category and $Q$ its Ext-quiver. Then there exists (an explicitely given) set of relations $R$ on $Q$ such that we have an equivalence of categories \[ e: \A \to Q/R-mod\] such that $e(M) = \bigoplus_{\lambda \in X^+} Hom_{\A}(P(\lambda),M)$ as graded vector spaces.
\end{thm}

\begin{lem} \cite{Assem} Let $M$ be an indecomposable representation of a finite quiver $Q$ which has no cyclic path. Then the number of composition factors of type $L(\mu)$ in $M$ are given by $dim M_{\mu}$ where $M_{\mu}$ is the vector space on the vertex $\mu$.
\end{lem}

A block $\Gamma$ of $X^+$ is a connected component of the Ext-quiver. Let $A_{\Gamma}$ be the full subcategory of objects of $\A$ such that all composition factors are in $\Gamma$ (also called a block). This gives a decomposition $\A = \bigoplus_{\Gamma} \A_{\Gamma}$ of full abelian subcategories. Every indecomposable module lies in a unique $A_{\Gamma}$ and all its simple submodules belong to $\Gamma$. Two irreducible representations $L(\lambda)$ and $L(\mu)$ are in the same block if and only if the weights $\lambda$ and $\mu$ define labelings with the same position of the labels $\times$ and $\circ$. The degree of atypicality is a block invariant, and the blocks $\Gamma$ of atypicality $r$ are in 1-1 correspondence with pairs of disjoint subsets of $\bbZ$ of cardinality $m-r$ resp. $n-r$.

\begin{thm} Assume $m,n \geq 2$. Then $GL(m|n)^{red}$ is not of finite type.
 \end{thm}

\begin{proof} This follows from the description of the Tannaka group generated by the irreducible elements in \cite[lemma 11.4]{Heidersdorf-Weissauer-tannaka}.  
\end{proof}

The statement also follows from the following lemma. This lemma should of course also hold for $m,n \geq 2$, but would require a more difficult argument. Let $Q$ denote the Ext-quiver of $\calR_{mn}$. Then there exists a system of relations $R$ on $Q$ such that $\calR_{mn} \simeq kQ/R-mod$.

\begin{lem} \label{thm:wildness} Assume $m,n \geq 3$. Then the problem of classifying indecomposable representations of non-vanishing superdimension is wild.
\end{lem}

\begin{proof}
We show that that the classication is wild for every maximally atypical block for $n \geq 3$. Any such block is equivalent to the maximal atypical block $\Gamma$ of $GL(n|n)$ \cite[Theorem 3.6]{Serganova-blocks} \cite{Brundan-Stroppel-4}. Hence we show that the problem is wild in $\Gamma$. By \cite[Corollary 5.15]{Brundan-Stroppel-2}, for any two irreducible modules $L(\lambda), L(\mu) \in {\calR}_n$ \[ \dim(Ext^1_{\calR_n}(L(\lambda),L(\mu))) = p_{\lambda,\mu}^{(1)} + p_{\mu,\lambda}^{(1)} \] for the Kazhdan-Lusztig polynomials \[ p_{\lambda,\mu}(q) = \sum_{i \geq 0} p_{\lambda,\mu}^{(i)} q^i.\] By \cite[Lemma 6.10]{Musson-Serganova} and \cite[Lemma 5.2]{Brundan-Stroppel-2} $p_{\lambda,\mu}^{(1)} \neq 0$ if and only if $\mu$ is obtained from $\lambda$ by interchanging the labels at the ends of one of the cups in the cup diagram of $\lambda$. For any $[\lambda] \in \Gamma$ with $\lambda_i > \lambda_{i+1}+1$ for $i=1,\ldots,n-1$ the cup diagram looks like 

\medskip
\begin{center}

\begin{tikzpicture}
 \draw (-6,0) -- (6,0);
\foreach \x in {} %vee
     \draw (\x-.1, .2) -- (\x,0) -- (\x +.1, .2);
\foreach \x in {} %wedge
     \draw (\x-.1, -.2) -- (\x,0) -- (\x +.1, -.2);
\foreach \x in {} %cross
     \draw (\x-.1, .1) -- (\x +.1, -.1) (\x-.1, -.1) -- (\x +.1, .1);
%\foreach \x in {1} %circle
     %\draw  node at (-2,0) [fill=white,draw,circle,inner sep=0pt,minimum size=6pt]{};
     %\draw  node at (2,0) [fill=white,draw,circle,inner sep=0pt,minimum size=6pt]{};
     %\draw  node at (3,0) [fill=white,draw,circle,inner sep=0pt,minimum size=6pt]{};

%%caps,cups
\draw [-,black,out=90, in=90](-5,0) to (-4,0);
\draw [-,black,out=90, in=90](-2,0) to (-1,0);
\draw [-,black,out=90, in=90](1,0) to (2,0);
\draw [-,black,out=90, in=90](4,0) to (5,0);
%\draw [-,black,out=270, in=270](-5,0) to (4,0);
%

\end{tikzpicture}

\end{center}
\medskip

The combinatorial rule from above shows that for every irreducible module $[\lambda]$ away from the diagonal $dim Ext^1([\lambda], [\mu_i]) = dim Ext^1([\mu_i], [\lambda]) = 1$ for exactly $2n$ different modules $\mu_i$ and $dim Ext^1([\lambda], \nu) = 0$ for any $\nu \neq \mu_i$. In particular for any vertex away from the diagonal consider the subquiver with vertices $[\lambda], [\mu_1], \ldots, [\mu_{2n}]$ with arrows corresponding to $dim Ext^1([\mu_i], [\lambda]) = 1$ and no arrows from $[\lambda]$ to any $[\mu_i]$ (so that $[\lambda]$ becomes a sink) (picture for $n=3$): 
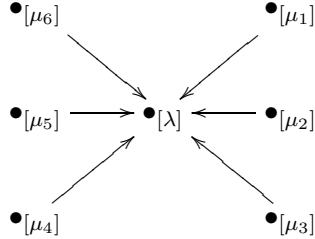
\[ \xymatrix{ \bullet_{[\mu_6]} \ar[dr] & & \bullet_{[\mu_1]} \ar[dl] \\ \bullet_{[\mu_5]} \ar[r] & \bullet_{[\lambda]} & \bullet_{[\mu_2]} \ar[l] \\ \bullet_{[\mu_4]} \ar[ur] & & \bullet_{[\mu_3]} \ar[ul].} \] Since this subquiver has no path of length $>1$, it embeds fully into $k(Q)/R$.  The classification of indecomposable representations of the $r$-subspace quiver is wild for $r \geq 5$ \cite[Section 10]{Krause-quivers}. The superdimension formula of \cite{Weissauer-gl} \cite[Section 16]{Heidersdorf-Weissauer-tensor} shows that the superdimension is constant of alternating sign away from the diagonal: if $[\lambda]$ has superdimension $d$, the $[\mu_i]$ have superdimension $-d$. Hence an indecomposable representation of this subquiver will give an indecomposable representation in $\Gamma$ of non-vanishing superdimension if and only if \[ dim V_{[\lambda]} \neq \sum_{i=1}^{2n} dim V_{[\mu_i]} \ \ \ (*).\] We are done when we have shown that the classification of indecomposable representations with $(*)$ is wild. Fix the vertex $[\mu_{2n}]$ and consider an indecomposable representation of the $(2n-1)$-subspace quiver by specifying a vector space for the vertices $[\lambda], [\mu_1], \ldots, [\mu_{2n-1}]$ with injections $V_{[\mu_i]} \to V_{[\lambda]}$. If $dim V_{[\lambda]} \neq \sum_{i=1}^{2n-1} dim V_{[\mu_i]}$ we put $V_{[\mu_{2n}]} = 0$. If $dim V_{[\lambda]} = \sum_{i=1}^{2n-1} dim V_{[\mu_i]}$ we put $V_{[\mu_{2n}]} = k$ and choose some injection of $k$ into $V_{[\lambda]}$. This defines a bijection between the isomorphism classes of indecomposable representations of the $(2n-1)$-subspace quiver with a subset of the indecomposable representations of the $2n$-subspace quiver satisfying $(*)$.  
\end{proof}

As explained above this means that we should not try to determine $G^{red}$ for general $m,n$. Instead we should determine the tensor category generated by the irreducible elements in $\A/\NN$ and the corresponding reductive supergroup. In the $n=1$-case, where we have tame representation type, we determine the full quotient in this article (theorem \ref{thm:main}). We prove this by describing explicitly the image of an indecomposable representation $I$.

%\begin{thm} We have: \begin{align*} G^{red} = \begin{cases} Gl(m-1) \times Gl(1) \times Gl(1) \quad & G = Gl(m|1) \\  Sl(m-1) \times Gl(1) \times Gl(1)  & G = Sl(m|1), \ m \geq 3 \\   Gl(1) \times Gl(1) & G = Sl(2|1). \end{cases} \end{align*} \end{thm}
%
%This theorem should be understood in the super sense, i.e. $Rep(Gl(m|1))/\mathcal{N} \simeq Rep(Gl(m-1) \times Gl(1) \times Gl(1)) \otimes svec_k$. We prove this by describing explicitely the image of an indecomposable representation $I$.

%In the $n=1$-case, where we have tame representation type, we show in this article:
%
%\begin{thm} We have: \begin{align*} G^{red} = \begin{cases} Gl(m-1) \times Gl(1) \times Gl(1) \quad & G = Gl(m|1) \\  Sl(m-1) \times Gl(1) \times Gl(1)  & G = Sl(m|1), \ m \geq 3 \\   Gl(1) \times Gl(1) & G = Sl(2|1). \end{cases} \end{align*} \end{thm}
%
%This theorem should be understood in the super sense, i.e. $Rep(Gl(m|1))/\mathcal{N} \simeq Rep(Gl(m-1) \times Gl(1) \times Gl(1)) \otimes svec_k$. We prove this by describing explicitely the image of an indecomposable representation $I$.

%%%%%%%%%%%%%%%%%%%%%%%%%%%%%%%%%%%

%%%%%%%%%%%%%%%%%%%%%%%%%%%%%%%%%%%%

\section{Determination of $G^{red}$ for $G = GL(m|1)$} \label{determination-of-envelope}

\subsection{Germoni's classification} Let $\Gamma$ be a singly atypical block of $\mathcal{R}_{m|1}$. Then $\Gamma$ is equivalent to the category of nilpotent $\Z$-graded finite-dimensional representations of $A$, the quotient of the free algebra on two generators $d^+$ and $d^-$ of respective degrees $+1$ and $-1$ by the relations $(d^+)^2 = (d^-)^2 = 0$ \cite[Section 5]{Germoni-sl}. Let $\tilde{\Gamma}$ be the following quiver \[  \xymatrix{ \ldots \ar[r] & \overset{(a-2)_t}{\bullet} & \overset{(a-1)_s}{\bullet} \ar[r] \ar[l] & \overset{a_t}{\bullet} & \overset{(a+1)_s}{\bullet} \ar[r] \ar[l] & \overset{(a+2)_t}{\bullet} & \ldots \ar[l]  \\  \ldots  & \overset{(a-2)_s}{\bullet} \ar[r] \ar[l] & \overset{(a-1)_t}{\bullet}  & \overset{a_s}{\bullet} \ar[r] \ar[l] & \overset{(a+1)_t}{\bullet} & \overset{(a+2)_s}{\bullet} \ar[r] \ar[l] & \ldots } \]

Then there is a bijection between the isomorphism classes of non-projective indecomposable representations of $\tilde{\Gamma}$ and the finite connected subquivers $I$ of $\tilde{\Gamma}$ other than $\overset{a_s}{\bullet}$ (which would give the same indecomposable representations as $\overset{a_t}{\bullet}$). Hence if we parametrize the irreducible representations in $\Gamma$ by $\Z$, any interval $[a,b]$ gives rise to two non-projective indecomposable representations which we call $I^+[a,b]$ and $I^-[a,b]$. These are also known as zigzag-modules \cite{Goetz-Quella-Schomerus}\cite{Su-classification-singly} due to the following graphical interpretation of their Loewy structure. With $I^+[a,b]$ we denote the indecomposable module attached to $[a,b]$ with $L(b)$ in the top \[
I^+[a,b] = \xymatrix{ &  \overset{a+1}{\bullet} \ar@{-}[dr] & & \overset{a+3}{\bullet} \ar@{-}[dr] & & \overset{b}{\bullet} \\ \underset{a}{\bullet} \ar@{-}[ur] & & \underset{a+2}{\bullet} \ar@{-}[ur] & & \underset{a+4}{\bullet} \ar@{-}[ur] & } \] 

where the composition factors in the top are those in the top row and the composition factors in the socle are in the lower row. The twisted dual of $I^+[a,b]$ is denoted by $I^-[a,b] = I^+[a,b]^*$ \[ I^-[a,b] = 
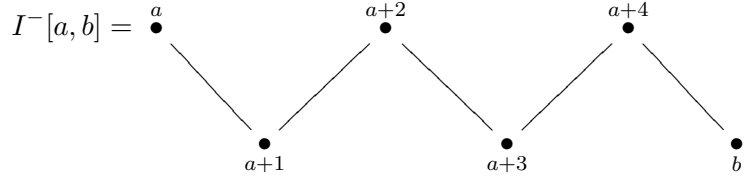
\xymatrix{  \overset{a}{\bullet} \ar@{-}[dr] & & \overset{a+2}{\bullet} \ar@{-}[dr] & & \overset{a+4}{\bullet} \ar@{-}[dr]& \\ & \underset{a+1}{\bullet} \ar@{-}[ur] & & \underset{a+3}{\bullet} \ar@{-}[ur] & & \underset{b}{\bullet}  }\] with $L(b)$ in the socle.

\begin{bsp} For the interval $[a,a]$ we have \[ I^+[a,a] = I^-[a,a] = L(a).\] For the interval $[a,a+1]$ we get 
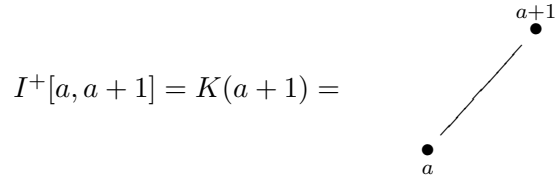
\[ \xymatrix@R-7mm{ & & \overset{a+1}{\bullet} \\ I^+[a,a+1] = K(a+1) =  & & \\  &  \underset{a}{\bullet} \ar@{-}[uur] & }  \] and 
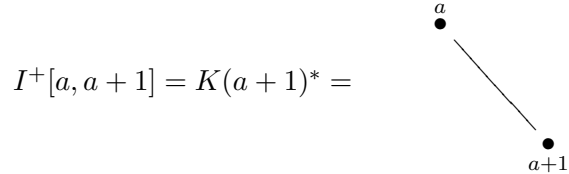
\[ \xymatrix@R-7mm{ &  \overset{a}{\bullet}  \ar@{-}[ddr] & \\ I^+[a,a+1] = K(a+1)^* =  & & \\  & &   \underset{a+1}{\bullet}   }  \]
\end{bsp}

\subsubsection{Superdimensions} The superdimension of any Kac module is zero for any type I Lie superalgebra. The superdimension is additive in short exact sequences, hence we obtain $sdim L(a) = - sdim L(a+1)$ for any atypical weight $a$.
  
\begin{cor} The superdimension of the indecomposable modules $I^+[a,b]$ and $I^-[a,b]$ is given by 
\begin{align*} sdim I^+[a,b] & = sdim I^-[a,b] \\ &  = \sum_{i=0}^{b-a} (-1)^i sdim L(a+i)  = \begin{cases} sdim L(a) & b-a  \text{ even } \\ 0 & b-a \text{ odd}. \end{cases} \end{align*} 
\end{cor}

The only other remaining indecomposable modules are the projective covers $P(\lambda)$ of the atypical simple modules which have superdimension 0. 

\begin{cor} \label{thm:irreducible} The irreducible objects in $\A/\NN$ are up to isomorphism given by the \[ \{ I^+[a,b], \ I^-[a,b] \in \Gamma \ | , \ b-a \ \text{even}\} \] for all atypical blocks $\Gamma$.
\end{cor}

%%%%%%%%%%%%%%%%%%%%%%%%%%%%%%%%%%%%

%%%%%%%%%%%%%%%%%%%%%%%%%%%%%%%%%%

%%%%%%%%%%%%%%%%%%%%%%%%%%%%%%%%%

\subsection{Mixed tensors} We first determine the contribution from the irreducible modules. The subcategory $T \subset \calR_{m|n}$ ($m\geq n$) of mixed tensors is the pseudoabelian full subcategory of objects, which are direct summands in a tensor product $V^{\otimes r} \otimes (V^{\vee})^{\otimes s}$ for some $r,s$. We have the equivalence \[ T/\NN \simeq Rep(GL(m-n))\] of tensor categories by \cite[Theorem 10.2]{Heidersdorf-mixed-tensors}. The $I^+[a,b]$ and $I^-[a,b]$ are never mixed tensors for $b > a$ \cite[Corollary 8.8]{Heidersdorf-mixed-tensors}. The indecomposable objects in $T$ are parametrized by $(m|n)$-cross bipartitions  $(\lambda^L,\lambda^R)$ \cite[Theorem 8.7.6]{Comes-Wilson} and we denote the corresponding indecomposable element by $R(\lambda^L, \lambda^R)$. 

\begin{thm} \label{thm-mixed} a) \cite[Lemma 10.1 and Proposition 10.3]{Heidersdorf-mixed-tensors}  The superdimension of $R(\lambda^L, \lambda^R)$ is non-zero if and only if $l(\lambda^L) + l(\lambda^R) \leq m-1$. In this case $R(\lambda^L, \lambda^R)$ is irreducible.

b) \cite[Theorem 8.2]{Heidersdorf-mixed-tensors}  Every irreducible atypical representation in $\mathcal{R}_{m|1}$ can be written in the form \[ L(\lambda) \cong Ber^{s(\lambda)} \otimes R(\lambda^L,\lambda^R)\] for some unique integer $s(\lambda)$ and a unique $(m|n)$-cross bipartition $(\lambda^L,\lambda^R)$ satisfying $l(\lambda^L) + l(\lambda^R) \leq m-1$.
\end{thm}

\begin{bem}The number $s(\lambda)$ is described in the proof of \cite[Theorem 8.2]{Heidersdorf-mixed-tensors} in terms of the weight diagram of $\lambda$. Its precise description is not needed in the present article.
\end{bem}

We now consider the tensor subcategory in $\mathcal{T}_{m|1}/\NN$ generated by the images of the irreducible modules. It is again a super tannakian category because it inherits the Schur finiteness of $\mathcal{T}_{m|1}/\NN$ (and then apply theorem \ref{deligne-super}). It is clear from theorem \ref{thm-mixed} that the reductive supergroup corresponding to this subcategory is $GL(1) \times GL(m-1)$. Indeed the element \[ Ber^{s(\lambda)} \otimes R(\lambda^L,\lambda^R) \] (where $\lambda^L = (\lambda^L_1,\ldots,\lambda^L_s,0,\ldots), \ \lambda_s^L > 0$, $\lambda^R = (\lambda^R_1,\ldots,\lambda^R_t,0,\ldots), \ \lambda_t^R > 0$, $t + s \leq m-1$)) corresponds to the representation \[ det^{s(\lambda)} \otimes L'(wt(\lambda^L,\lambda^R))\] for the irreducible $GL(m-1)$-representation \[ L'(wt(\lambda^L,\lambda^R)) = L( \lambda_1^L, \ldots, \lambda^L_s, 0, \ldots, 0, - \lambda^R_t,\ldots ,\lambda_1^R )\] defined in \cite{Comes-Wilson} \cite[Section 10]{Heidersdorf-mixed-tensors} and the determinant representation $det$.

%%%%%%%%%%%%%%%%%%%%%%%%%%%%%%%%%%%

%%%%%%%%%%%%%%%%%%%%%%%%%%%%%%

\subsection{Indecomposable modules} We therefore understand the tensor product decomposition between the irreducible representations in $\mathcal{T}_{m|1}$. In order to know the tensor product decomposition between any two irreducible representations of $\mathcal{T}_{m|1}/\NN$, it remains to compute the tensor product decompositions of the form $I^{\pm}[a,b] \otimes I^{\pm}[a',b']$ (where $a,b$ and $a',b'$ could be in different blocks) up to superdimension 0. This will be a complicated reduction to the case of irreducible representations.

\subsection{Cohomological tensor functors} 

We now define cohomological tensors $\mathcal{T}_{m|1} \to \mathcal{T}_{m-1|0}$. These were first defined in \cite{Duflo-Serganova} and then later refined in \cite{Heidersdorf-Weissauer-tensor}. We define a similar refined version in the $\mathcal{T}_{m|1}$-case, but it can be easily extended to the general $\mathcal{T}_{m|n}$-case \cite{Heidersdorf-cohomological-2}. For any $x \in X = \{ x \in \g_1 \ | \ [x,x] = 0 \}$ and any representation $(V,\rho)$, the operator $\rho(x)$ defines a complex since $\rho(x) \circ \rho(x)$ is zero, and we define $V_x = ker(\rho(x)) / im(\rho(x))$. By \cite{Duflo-Serganova} \cite{Serganova-kw} this defines a tensor functor and $V_x \in \mathcal{T}_{m-1|0}$. We fix the following element $x \in X$ \[ x = \begin{pmatrix} 0 & y \\ 0 & 0 \end{pmatrix} \in \mathfrak{gl}(m|1) \text{ for } y = \begin{pmatrix} 0  \\ 0 \\ \ldots  \\ 1 \end{pmatrix}. \] and denote the corresponding tensor functor $V \to V_x$ by $DS: \mathcal{T}_{m|1} \to \mathcal{T}_{m-1|0}$.

\subsubsection{Parity considerations} If $V$ is in $\calR_{m|1}$, $DS(V)$ may not be in $\calR_{m-1|0}$. We need to study this more closely. We embed $GL(m-1|0)$ as an upper block matrix in $GL(m|1)$. More precisely we fix the embedding \begin{align*} GL(m-1|0) \times GL(1|1) & \to GL(m|1) \\ A \times \begin{pmatrix} a & b \\ c & d \end{pmatrix} & \mapsto \begin{pmatrix} A & 0 & 0 \\ 0 & a & b \\ 0 & c & d \end{pmatrix}. \end{align*}
Recall that we defined the categories $\calR_{m|n}$ in section \ref{sec:preliminaries}. We now fix the morphism $\epsilon: \Z/2\Z \to GL(m) \times GL(n)$ which maps $-1$ to $diag(E_m,-E_n) = \epsilon_{m|n}$. Then $\mathcal{R}_{m|1}$ is the full subcategory of objects $V$ such that $p_V = \rho(\epsilon_{m|1})$. If we use the embedding above, we obtain $\epsilon_{m|1} = \epsilon_{m-1|0} \epsilon_{1|1}$. If we restrict a representation $V \in \calR_{m|1}$ to $GL(m-1)$, we get the decomposition \[ V|_{GL(m-1)} = V^+ \oplus V^- \] where \begin{align*} V^+ & = \{ v \in V \ | \ \rho(\epsilon_{1|1}) (v) = v \} \\  V^- & = \{ v \in V \ | \ \rho(\epsilon_{1|1}) (v) = -v \}.\end{align*} Since $\rho(x)$ is an odd morphism \[\rho(x): V^{\pm} \to V^{\mp}.\] Hence $\rho(x)$ induces the even morphism \[ \rho(x): V^{\pm} \to \Pi V^{\mp}.\] We use the notation $\partial$ for $\rho(x)$.

\subsubsection{$\Z$-grading}  We equip $DS(V)$ with a $\Z$-grading as in \cite[Section 3]{Heidersdorf-Weissauer-tensor}. Although $DS(V)$ is in $\mathcal{T}_{m-1|0}$, it still has an action of the torus of diagonal matrices in $GL(m|1)$. Let $V = \bigoplus_{\lambda}V_{\lambda}$ be the weight decomposition and $v = \sum_{\lambda} v_{\lambda}$ in $V$. An easy calculation shows the following lemma.

\begin{lem} The following holds for $\partial = \rho(x)$: a) $\partial(V_{\lambda}) \subset V_{\lambda + \mu}$ for the odd simple root $\mu = \epsilon_{n} - \epsilon_{m+1}$, b) $\partial v = 0$ if and only if $\partial v_{\lambda} = 0$ for all $\lambda$ and c) $v = \partial w$ if and only if $v_{\lambda} = \partial w_{\lambda}$ for all $\lambda$.
\end{lem}

Hence $DS(V)$ has a weight decompositon with respect to the weight lattice of $\mathfrak{gl}(m|1)$. The weight decomposition with respect to $\mathfrak{gl}(m-1|0)$ is obtained by restriction. The kernel of this restriction map consists of the multiples $\Z \mu$. Hence $DS(V)$ can be endowed with the weight structure coming from the $\mathfrak{gl}(m|1)$-module $V$. This weight decomposition induces then on $DS(V)$ a decomposition \[ DS(V) = \bigoplus_{l \in \Z} DS(V)_l.\] Now let $H_s$ denote the torus in the diagonal matrices of elements of the form $(1,\ldots,1,t^{-1})$ and denote by $V_l$ the eigenspace  of $V$ where $H_s$ acts by the eigenvalue $t^l$. Another easy calculation shows the next lemma.

\begin{lem} If $v \in V$ satisfies $v \in V_{\lambda}$ and $v \in V_l$, then $\rho(x)v \in V_{l+1}$.
\end{lem}

Since $\partial V_{\lambda} \subset V_{\lambda + \mu}$ we obtain a complex \[ \xymatrix{ \ldots \ar[r]^{\partial} & V_{\lambda} \ar[r]^{\partial} & \Pi V_{\lambda+ \mu} \ar[r]^{\partial} & V_{\lambda + 2\mu} \ar[r]^{\partial} & \ldots }\] which we can write by the last lemma as \[ \xymatrix{ \ldots \ar[r]^{\partial} & V_{l} \ar[r]^{\partial} & \Pi V_{l+1} \ar[r]^{\partial} & V_{l+2} \ar[r]^{\partial} & \ldots }.\]  We denote the cohomology of this complex by $H^l(V)$. Then $DS(V)_l = \Pi^l (H^l(V))$ and we obtain a direct sum decomposition of $DS(V)$ into $GL(m-1)$-modules \[ DS(V) = \bigoplus_l \Pi^l(H^l(V)).\] This extra structure is very important since it carries a lot more information then just the $\Z_2$-graded version of $DS(L)$ in $\mathcal{T}_{m-1|0}$.

\begin{lem} A short exact sequence in $\calR_{m|1}$ gives rise a long exact sequence for $H^l$.
\end{lem}

\medskip

We denote by $\sigma$ the automorphism of $\mathfrak{gl}(m|1)$ defined by $\sigma(x) = - x^T$ where $()^T$ denotes the supertranspose. Then $\sigma(x)$ is still a nilpotent element in $\g_1$. The corresponding tensor functor is denoted by $DS_{\sigma}: \mathcal{T}_{m|1} \to \mathcal{T}_{m-1|0}$. We can copy the arguments from above and endow $DS_{\sigma}(V)$ with a $\Z$-grading $DS_{\sigma}(V) = \bigoplus_l \Pi (H^l_{\sigma}(V))$. $DS$ and $DS_{\sigma}$ behave in the same way on irreducible objects, but differ on indecomposable elements, see lemma \ref{thm:kernel}.

\begin{cor} For any representation $V$ $DS(V)$ and $DS_{\sigma}(V)$ are $\Z$-graded objects.
\end{cor}

Another way of looking at this $\Z$-grading is the following. The collection of cohomology functors $H^i: \calR_{m|1} \to \calR_{m-1|0}$ for $i\in \mathbb Z$
defines a tensor functor
\[  H^\bullet: \calR_{m|1} \to Gr^\bullet(\calR_{m-1|0}) \ \]
to the category of $\mathbb Z$-graded objects in $\calR_{m-1|0}$.
Using the parity shift functor $\Pi$, this functor can be extended to a tensor functor
$$  H^\bullet: \mathcal{T}_{m|1} \to Gr^\bullet(\mathcal{T}_{m-1|0}).$$

As in \cite[Theorem 4.1]{Heidersdorf-Weissauer-tensor} we conclude from the support variety calculations in \cite[Section 3]{BKN-2} the following lemma

\begin{lem}\label{thm:kernel} The kernel of $DS$ is $\calC^-$ and the kernel of $DS_{\sigma}$ is $\calC^+$.
\end{lem}

\subsection{Cohomology computations} \label{core} We can now calculate $DS(M)$ for any $M$ of nonvanishing superdimension. By \cite{Duflo-Serganova} \[ DS(L(\lambda)) = m L^{core} \oplus m' \Pi L^{core} \ \in \mathcal{T}_{m-1|0}\] for atypical $L(\lambda)$ where $L^{core} \in \mathcal{R}_{m|1}$ is the irreducible $GL(m-1)$-representation obtained from $L(\lambda)$ by replacing the single $\vee$ in the weight diagram by a $\wedge$. We recall from \cite[Section 10]{Heidersdorf-mixed-tensors} that we have a commutative diagram \[ \xymatrix{ & Rep(GL_{m-1}) \ar[dd]^{F_{m-1|0}} \ar[dl]^{F_{m|1}}  \\ \calR_{m|1} \ar[dr]^{DS} & \\ & \mathcal{R}_{m-1|0} } \] where $Rep(GL_{m-1})$ denotes Deligne's category to the parameter $\delta = m-1$. For this note that the image of any mixed tensor under $DS$ is in $\mathcal{R}_{m-1|0}$. We obtain the same diagram if we replace $DS$ by $DS_{\sigma}$ since $DS_{\sigma}$ sends the standard to the standard representation.

\begin{lem}\label{mainthm} Let $L(\lambda)$ be atypical irreducible. Then \[ DS(L(\lambda)) = DS_{\sigma} (L(\lambda)) = \begin{cases} L^{core} \ & \text{ if } p(\lambda) \text{ even.} \\ \Pi L^{core} \ & \text{ if } p(\lambda)  \text{ odd. } \end{cases} \]
\end{lem}

\begin{proof} If $L = R(\lambda)$ is the unique mixed tensor in a block $\Gamma$, then $DS(L) = L'(wt(\lambda))$ by the commutativity of the diagram above. Hence the multiplicity of $L^{core}$ is 1 (i.e. $ m =1$ and $m' =0$). Any other irreducible representation is a Berezin twist of $L$ and $DS(Ber^i) = \Pi^i L(i,\ldots,i)$. The Berezin twist does not change the multiplicity. The superdimension of $L^{core}$ is positive and the superdimension of $\Pi L^{core}$ is negative. Since the superdimension of $L(\lambda)$ is positive if and only if $p(\lambda)$ is even, the result follows.

\end{proof}

\begin{bem} $L^{core}$ is the same as $DS(R(\lambda^L,\lambda^R)) =L'(wt(\lambda^L,\lambda^R))$ for the unique irreducible mixed tensor $R(\lambda^L,\lambda^R)$ in the block \cite[Lemma 8.1]{Heidersdorf-mixed-tensors}.
\end{bem}

\begin{lem} \label{DS-indecomposable} We have the following equalities \begin{align*} DS(I^+[a,b]) & = DS(L(b)),  & DS_{\sigma} (I^+[a,b])  & = DS(L(a)) \\ DS(I^-[a,b]) & = DS(L(a)),  & DS_{\sigma}(I^-[a,b]) & = DS_{\sigma}(L(b)) \end{align*} of $\Z$-graded objects in $\mathcal{T}_{m-1|0}$.
\end{lem}

\begin{proof} Every $I^+$ or $I^-$-module can be written in two ways, either as an extension of an object in $\mathcal{C}^+$ or of an object in $\mathcal{C}^-$. The module $I^+[a,b]$ is realized by the nontrivial extension \[ 0 \to K \to I^+[a,b] \to L(b) \to 0 \] where $K \in \mathcal{C}^-$ is costandard object with graded pieces $K(a+1)^*, K(a+3)^*,\ldots, K(b-1)^*$. Similarly projection onto $L(a)$ gives the extension \[ 0 \to K' \to I^+[a,b] \to L(a) \to 0 \] where $K' \in \mathcal{C}^+$ with graded pieces $K(a+2), K(a+4),\ldots, K(b)$. Likewise $I^-[a,b]$ is the indecomposable module corresponding to \[ 0 \to L(a) \to I^-[a,b] \to \tilde{K} \to 0 \] where $\tilde{K} \in \mathcal{C}^+$ with graded pieces $K(a+2)^*, K(a+4)^*,\ldots, K(b)^*$ and \[ 0 \to L(b) \to I^-[a,b] \to \tilde{K}' \to 0 \] where $\tilde{K}' \in \mathcal{C}^-$ with graded pieces $K(a+1), K(a+3), \ldots, K(b-1)$.

We apply $DS$. By lemma \ref{thm:kernel} the kernel of $DS$ is $\calC^-$. Hence taking the long exact cohomology sequence arising from \[ 0 \to K \to I^+[a,b] \to L(b) \to 0 \] gives, using $H^i(K) = 0$ for all $i$ and that the cohomology of every irreducible representation is concentrated in one degree, that \[ H^i (I^+[a,b]) \simeq H^i(L(b))\] for all $i$. Using $DS(V) = \bigoplus_l \Pi^l H^l(V)$ we obtain $DS(I^+[a,b]) = DS(L(b))$ as a $\Z$-graded object in $\mathcal{T}_{m-1|0}$ or $\Z \times \mathcal{T}_{m-1|0}$. The exact sequence \[ 0 \to K' \to I^+[a,b] \to L(a) \to 0 \] gives, using $Ker(DS_{\sigma}) = \calC^+$, the identification \[ H_{\sigma}^i(I^+[a,b]) \simeq H_{\sigma}^i(L(a)) \] for all $i$. Hence $DS_{\sigma}(I^+[a,b]) = DS_{\sigma}(L(a))$ as a $\Z$-graded object in $\mathcal{T}_{m-1|0}$ or $\Z \times \mathcal{T}_{m-1|0}$. In the same way we conclude \begin{align*}   H^i(I^-[a,b] ) & \simeq H^i(L(a)) \text{ for all } i \\ H_{\sigma}^i(I^-[a,b]) & \simeq H_{\sigma}^i(L(b)) \text{ for all } i. \end{align*} 
\end{proof}

\subsection{Tensor products up to superdimension zero} We now calculate the tensor product decomposition of $I^{\pm}[a,b] \otimes I^{\pm}[a',b']$ up to superdimension zero using lemma \ref{DS-indecomposable}. Since D$S$ sends every indecomposable module in the same block to the same irreducible $GL(m-1)$-representation, it might be a bit surprising that we can use $DS$ for this. It is crucial here that $DS(I^{\pm}[a,b])$ is a $\Z$-graded object and that the intersection of the fibre of a $\Z$-graded irreducible object in $\mathcal{T}_{m-1|0}$ under $DS$ respectively $DS_{\sigma}$ consists of one indecomposable representation.

\medskip
\textit{Notation.} So far we worked within a fixed block $\Gamma$ so that the notation $I^+[a,b]$ was unambigious. In the following tensor product decomposition we need to specify the block of $I^+[a,b]$. There is a natural bijection between atypical blocks of $\mathcal{R}_{m|1}$ and irreducible representations of $GL(m-1)$ given by $\Gamma \mapsto L^{core}$ as in section \ref{core}. We write $\Gamma_{\lambda'}$ for the block corresponding to the dominant integral weight $\lambda'$ of $GL(m-1)$, and we write $I^+[a,b]_{\lambda'}$ instead of $I^+[a,b]$. For $a=b$ we write then $L(a)_{\lambda'}$ for $I^{\pm}[a,a]_{\lambda'}$. For weights $\lambda',\mu'$ of $GL(m-1)$ and the corresponding irreducible representations $L'(\lambda'), L'(\mu') \in \mathcal{R}_{m-1|0}$ we write $c_{\lambda'\mu'}^{\nu'}$ for the coefficients in $L'(\lambda') \otimes L'(\mu') \cong \bigoplus_{\nu'} c_{\lambda'\mu'}^{\nu'} L'(\nu')$.

\medskip
\textit{Normalization.} So far the parametrization of the irreducible objects in a block was not specified. We use the block equivalence of an atypical block with the principal block in $\mathcal{R}_{1|1}$ \cite{Serganova-blocks}\cite[Section 6.3]{BKN-complexity-gl} to fix this parametrization. We denote by $L(0)$ the irreducible representation in $\Gamma_{\lambda'}$ corresponding to the trivial representation $\Eins \in \mathcal{R}_{1|1}$ under this equivalence. Then $H^a(L(a)) \cong L^{core}$ and is zero in other degrees.

%We write $c_{a,b}^c$ for the coefficients in a $Gl(m-1)$ tensor product $L(a) \otimes L(b) = \bigoplus c_{a,b}^c L(c)$.

\begin{lem} \label{thm:tensor-product} Up to superdimension $0$ summands we have the following decompositions
\begin{align*} I^+[a,b]_{\lambda'} \otimes I^+[a',b']_{\mu'}  & = \bigoplus_{\nu'} c_{\lambda' \mu'}^{\nu'}  I^+[a+a',b+b']_{\nu'}  \\  I^-[a,b]_{\lambda'} \otimes I^-[a',b']_{\mu'}  & = \bigoplus_{\nu'} c_{\lambda' \mu'}^{\nu'} I^-[a+a',b+b'] \\  I^+[a,b]_{\lambda'} \otimes I^-[a',b']_{\mu'} & = \begin{cases} \bigoplus_{\nu'} c_{\lambda' \mu'}^{\nu'}   L(b+a')_{\nu'} & \text{ if } b-a = b'-a' \\ \bigoplus_{\nu'} c_{\lambda' \mu'}^{\nu'} I^+[a+b', b+a'] & \text{ if } b-a > b'-a' \\ \bigoplus_{\nu'} c_{\lambda' \mu'}^{\nu'}  I^-[b+a',a+b'] & \text{ if } b-a < b'-a'. \end{cases} \end{align*} 
\end{lem}

\begin{proof} We consider the tensor product $I^+[a,b]_{\lambda'} \otimes I^+[a',b']_{\mu'}$. Under $DS$ the two modules map to two irreducible elements of $\Z \times \mathcal{T}_{m-1|0}$, namely $ b \times L'(\lambda')$ and $b' \times L'(\mu')$. Their tensor product is given by the Littlewood-Richardson-rule \[  (b \times L'(\lambda') ) \otimes (b' \times L'(\mu')) = \bigoplus_{\nu'}  (b + b') \times c_{\lambda' \mu'}^{\nu'} L'(\nu').\] Note that not only $DS: \mathcal{T}_{m|1} \to \mathcal{T}_{m-1|0}$ is a tensor functor, but that also the induced functor $DS: \mathcal{T}_{m|1} \to \Z \times \mathcal{T}_{m-1|0}$ is compatible with the tensor product by the K\"unneth formula for the cohomology using that the cohomology of every indecomposable object is concentrated in one degree. 

Under the tensor functor $DS_{\sigma}$ the two indecomposable objects map again to two irreducible objects in $\Z \times \mathcal{T}_{m-1|0}$. Now $DS_{\sigma}$ agrees with $DS$ on irreducible representations, and on the indecomposable modules $I^+[a,b]_{\lambda'}$ and $I^+[a',b']_{\mu'}$ the two functors differ by a $\Z$-shift: $I^+[a,b]_{\lambda'}$ maps to $\tilde{a} \times L'(\lambda')$ and $I^+[a',b']_{\mu'}$ maps to $\tilde{a}' \times L'(\mu')$. For the tensor product we obtain \[  (\tilde{a} \times L'(\lambda') ) \otimes (\tilde{a}' \times L'(\mu')) = \bigoplus_{\nu'}  (\tilde{a} + \tilde{a}') \times c_{\lambda' \mu'}^{\nu'} L'(\nu').\] Hence the tensor products just differs by a $\Z$-shift by $(\tilde{b} - \tilde{a}) - (\tilde{a} - \tilde{a}')$. 

We look at the fibre of a summand $(b + b') \times L'(\nu')$ under $DS$ and $DS_{\sigma}$. The fibre under $DS$ in the block $\Gamma_{\nu'}$ consists of \begin{align*} & L(b + b') \\ & I^+[c',b + b'] \ \text{ for any } c' \leq b + b' \\ & I^-[b + b', c'] \ \text{ for any } c' \geq b + b'.  \end{align*} 

The fibre of the corresponding summand $(a + a') \times  L'(\nu')$ under $DS_{\sigma}$ in the block $\Gamma_{\nu'}$ consists of \begin{align*} & L(a + a') \\ & I^+[a+a',c'] \ \text{ for any } c' \geq a + a' \\ & I^-[c',a + a'] \ \text{ for any } c' \leq a + a'.\end{align*} 

The only indecomposable representation in the intersection of the two fibres is $I^+[a+a',b+b']$. Hence this indecomposable module will appear as a direct summand in the tensor product decomposition.

The case $I^-[a,b]_{\lambda'} \otimes I^-[a',b']_{\mu'}$ can be reduced to the first case using $(A \otimes B)^* \cong A^* \otimes B^*$ for the twisted dual $()^*$.

The $I^+[a,b]_{\lambda'} \otimes I^-[a',b']_{\mu'}$ tensor product maps under $DS$ to \[ (b \times L'(\lambda') ) \otimes (a' \times L'(\mu')) = \bigoplus_{\nu'}  (b + a') \times c_{\lambda' \mu'}^{\nu'} L'(\nu'). \] and under $DS_{\sigma}$ to \[ (a \times L'(\lambda') ) \otimes (b' \times L'(\mu')) = \bigoplus_{\nu'}  (a + b') \times c_{\lambda' \mu'}^{\nu'} L'(\nu'). \] Hence these two tensor products differ by a $\Z$-shift by $(b - a') - (a - a')$. The $DS$ fibre of a summand $(b + a') \times  L'(\nu')$ in the block $\Gamma_{\nu'}$ consists of \begin{align*} & L(b+ a') \\ & I^+[c,b + a'] \ \text{ for any } c \leq b + a' \\ & I^-[b + a', c] \ \text{ for any } c \geq b + a'.  \end{align*} The $DS_{\sigma}$ fibre of the summand $(a + b') \times  L'(\nu')$ in the block $\Gamma_{\nu'}$ consists of \begin{align*} & L(a+b') \\ & I^+[a+b',c'] \ \text{ for any } c' \geq a+b' \\ & I^-[c',a+b'] \ \text{ for any } c' \leq a + b'.  \end{align*} 

If $b+a' = a +b'$ (or equivalently $b-a = b'-a'$), then the intersection of the fibres contains only the object $L(b+a')$. If $b + a' > a + b'$, the intersection consists of the object $I^+[a+b',b+a']$ and for $b+a' < a + b'$ of the object $I^-[b+a',a+b']$.    
\end{proof}

\begin{bem} In the $\mathfrak{sl}(2|1)$-case explicit formulas for the tensor products between any two indecomposable objects are already known \cite{Goetz-Quella-Schomerus}.  
%Up to superdimension $0$ summands these are given by the following rules:
%\begin{itemize} 
%\item $Z^{2p_1 + 1} (j_1) \otimes Z^{2p_2 + 1}(j_2) = Z^{2(p_1 + p_2) + 1}(j_1 + j_2)$
%\item $\bar{Z}^{2p_1 + 1} (j_1) \otimes \bar{Z}^{2p_2 + 1}(j_2) = \bar{Z}^{2(p_1 + p_2) + 1}(j_1 + j_2)$
%\item $Z^{2p_1 + 1} (j_1) \otimes \bar{Z}^{2p_2 + 1}(j_2) = \bar{Z}^{2(p_2 - p_1) + 1}(j_1 + j_2 - p_1)$ for $p_1 \leq p_2$
%\item $Z^{2p_1 + 1} (j_1) \otimes \bar{Z}^{2p_2 + 1}(j_2) = Z^{2(p_1 - p_2) + 1}(j_1 + j_2 - p_2)$ for $p_2 \leq p_1$
%\end{itemize}
\end{bem}

\subsection{The pro-reductive envelope} In this section we prove the following theorem.

\begin{thm} \label{thm:main}The quotient $\mathcal{T}_{m|1}/\NN$ is equivalent as a tensor category to the super representations of $GL(m-1) \times GL(1) \times GL(1)$, i.e. \[ \mathcal{T}_{m|1}/\NN \simeq Rep(\Z/2\Z \ltimes (GL(m-1) \times GL(1) \times GL(1)), (-1,e)). \] 
\end{thm}

Before the proof we review the decomposition of lemma \ref{thm:tensor-product}. The shift $I^+[a,b] \otimes I^+[a',b'] = \bigoplus  I^+[a+a',b+b']$ seems to suggest that an extra $GL(1)$-factor comes into play: The length $b-a$ behaves like the determinant power $det^{b-a}$. Indeed the shift above reads symbolically $length(I^+[a,b]) \otimes length(I^+[a',b']) = length(I^+[a+a',b+b'])$, or $det^{b-a} \otimes det^{b'-a'} = det^{b+b' - (a + a')}$.  Note however that the length is always positive, so we only have positive determinant powers. The same is true for the $I^- \otimes I^-$ tensor product. The $I^+[a,b]_{\lambda'} \otimes I^-[a',b']_{\mu'}$ tensor product decomposition then suggests that the $I^+[a,b]$ should correspond to determinant powers $det^{b-a}$ and the $I^-[a,b]$ to negative determinant powers $det^{-(b-a)}$.  

\begin{proof} If $I_1,I_2 \in \mathcal{T}_{m|1}$ are two indecomposable objects with positive superdimension, then every non-negligible summand in the decomposition $I_1 \otimes I_2$ has positive superdimension. Indeed by lemma \ref{DS-indecomposable} and lemma \ref{mainthm} using that that $DS(L(\lambda)) \cong  L^{core}$ if $p(\lambda)$ is even, each summand in $DS(I_1) \otimes DS(I_2)$ has positive superdimension. 

Now we consider the full subcategory in $\mathcal{T}_{m|1}$ of direct summands in iterated tensor products of indecomposable representations of superdimension $\geq 0$.  This category contains every indecomposable object up to a parity shift. It is closed under tensor products and duals and defines a pseudo-abelian tensor subcategory $\mathcal{T}_{m|1}^+$ of $\mathcal{T}_{m|1}$. Hence the quotient $\mathcal{T}_{m|1}^+/\NN$ is defined and a super tannakian category by proposition \ref{prop:fundamental}. By \cite[Theorem 7.1]{Deligne-Festschrift} any tensor category $\mathcal{A}$ (in the sense of loc.cit) which satisfies $dim_{\mathcal{A}}(X) > 0$ for all objects $X \in \mathcal{A}$ is a  tannakian category. Hence $\mathcal{T}_{m|1}^+/\NN$ is the (ordinary) representation category of a reductive  group $G^{red,+}$. Then $\mathcal{T}_{m|1}/\NN$ is the category of super representations of $G^{red,+}$ and in particular $G^{red} = \Z/2\Z \ltimes G^{red,+}$. 

By \cite[Proposition 2.20]{Deligne-Milne} an algebraic group $G$ is connected if and only if there exists an object $X \in Rep(G)$ such that every object of $Rep(G)$ is isomorphic to a subquotient of $X^n$ for some $n \geq 0$. The tensor product decomposition in lemma \ref{thm:tensor-product} shows then that there is no such $X$ and $G^{red,+}$ is therefore connected. We now use the preliminary considerations before the proof to define a homomorphism $\phi$ between the Grothendieck semirings $K_0^+(G^{red,+})$ and $K_0^+(GL(m-1) \times GL(1) \times GL(1))$. 

Clearly we should have for $L(\lambda) \cong Ber^{s(\lambda)} \otimes R(\lambda^L,\lambda^R)$ that \[ [ \ L(\lambda) \ ] \mapsto  [ \ \Eins \otimes det^{s(\lambda)} \otimes L'(wt(\lambda^L,\lambda^R)) \ ] \] in agreement with theorem \ref{thm-mixed}. By lemma \ref{thm:tensor-product} the $I^+[a,b]_{\lambda'} \otimes I^+[a',b']_{\mu'}$ tensor product decomposes in the same way as $I^+[a,a]_{\lambda'} \otimes I^+[a',a']_{\mu'}$ except for the additional $GL(1)$-shift with the length. If the irreducible representation $L(a)$ in the block $\Gamma_{\lambda'}$ is $L(\lambda) \cong Ber^{s(\lambda)} \otimes R(\lambda^L,\lambda^R)$, then we define for the classes of the $I^{\pm}$ in $K_0^+$ \begin{align*} [\ I^+[a,b]_{\lambda'} \ ] & \mapsto [ \ det^{b-a} \otimes det^{s(\lambda)} \otimes L'(wt(\lambda^R,\lambda^R)) \ ] \\ [\ I^-[a,b]_{\lambda'} \ ] & \mapsto [ \ det^{-(b-a)} \otimes det^{s(\lambda)} \otimes L'(wt(\lambda^R,\lambda^R)) \ ]\end{align*} We know by theorem \ref{thm-mixed} that this is compatible with the ring structure for the classes of irreducible modules $L(\lambda)$. It then follows from lemma \ref{thm:tensor-product} and the preliminary considerations before the proof that this defines a ring homomorphism. By definition this sends (classes of) irreducible representations to irreducible representations and is clearly bijective. Since we know that $G^{red,+}$ is a connected reductive group, we can use \cite[Theorem 1.2]{Kazhdan-Larsen-Varshavsky} that every homomorphism of Grothendieck semirings $\phi: K_0(H) \to K_0(G)$ (for $H, G$ connected reductive) which maps irreducible representations to irreducible, comes from a group homomorphism $\phi^*:H \to G$. If $\phi$ is an isomorphism, so is $\phi^*$ \cite[Theorem 1.2]{Kazhdan-Larsen-Varshavsky}. Hence the result follows.
\end{proof}

\begin{bem} It might be tempting to define a tensor functor between the two representation categories directly instead of passing to $K_0^+$. While one may pass to the skeletal subcategories in order to define the functor on objects and not just on isomorphism classes, it is a priori not clear that this is compatible with the associativity constraints.
\end{bem}

\begin{bem} The reason for considering $\mathcal{T}_{m|1}^+$ is to get rid of possible $OSp(1|2n)$-factors.
\end{bem}

\begin{bem} It is crucial here to consider the Grothendieck semiring, not the Grothendieck ring. Indeed if $G$ is connected, semisimple, and simply connected, then the Grothendieck $K_0(G)$ is isomorphic to $\Z[x_1,\ldots,x_r]$ where $r$ is the rank of $G$ \cite{Kazhdan-Larsen-Varshavsky} and encodes only the rank.
\end{bem}

\subsection{The special linear case} In the $SL(m|1)$-case, $Ber$ is trivial and does not generate an extra $GL(1)$-factor in the subgroup of $G^{red}$ generated by the irreducible representations. Otherwise the discussion is the same. Using $GL(m) \cong SL(m) \times GL(1)$ for $m \geq 2$ we get \begin{align*} G^{red} = \begin{cases} \mu_2 \ltimes ( \ SL(m-1) \times GL(1) \times GL(1) \ ) & \text{ for } m > 2\\ \mu_2 \ltimes ( \ GL(1) \times GL(1) \ ) & \text{ for } m=2.\end{cases} \end{align*}

\begin{bem} By \cite{vdJ-character} the character formula for an atypical representation of $GL(m|1)$ has the form of a character formula of $GL(m-1)$. Hence the superdimension of any irreducible representation equals up to a $(-1)^{p(\lambda)}$ factor the dimension of an irreducible $GL(m-1)$-representation. Our results give a conceptual explanation for this.
\end{bem}

%
%\cleardoublepage
%\pagestyle{plain}
%\phantomsection
%\addcontentsline{toc}{section}{Bibliography}
%\bibliographystyle{alpha}
%\bibliography{mybib}{}
%\bibliography{main-mixed-tensors}{}
%\include{bibliographie}

\end{document}